\documentclass[11pt]{amsart}
\usepackage{amsmath,amsthm,amsfonts,amssymb,times}
\usepackage{mathrsfs}  %% nice script fonts

\numberwithin{equation}{section}  %%% equations numbers (A.B)

\DeclareMathAlphabet{\curly}{U}{rsfs}{m}{n}  %% curly font

\theoremstyle{remark}

\theoremstyle{plain}

\newtheorem{lem}{Lemma}[section]
\newtheorem{thm}{Theorem}

\numberwithin{equation}{section}

\newcommand{\NN}{{\mathbb N}}

\newcommand{\E}{\mathbf E}
\makeatletter
\renewcommand{\pmod}[1]{\allowbreak\mkern7mu({\operator@font mod}\,\,#1)}
\makeatother

\newcommand{\bal}{\[\begin{aligned}}
\newcommand{\eal}{\end{aligned}\]}

\newcommand{\be}{\begin{equation}}
\newcommand{\ee}{\end{equation}}

\newcommand{\lam}{\ensuremath{\lambda}}

\renewcommand{\a}{\ensuremath{\alpha}}
\renewcommand{\b}{\ensuremath{\beta}}

\newcommand{\eps}{\ensuremath{\varepsilon}}

\renewcommand{\le}{\leqslant}

\renewcommand{\ge}{\geqslant}

\newcommand{\fl}[1]{{\ensuremath{\left\lfloor {#1} \right\rfloor}}}

\newcommand{\order}{\asymp}      %% order of magnitude
\renewcommand{\(}{\left(}
\renewcommand{\)}{\right)}

\newcommand{\pfrac}[2]{\left(\frac{#1}{#2}\right)}  %%% frac with paren

\begin{document}

\title{Extremal properties of product sets}
\author{Kevin Ford}
\date{}
\address{Department of Mathematics, 1409 West Green Street, University
of Illinois at Urbana-Champaign, Urbana, IL 61801, USA}
\email{ford@math.uiuc.edu}

\dedicatory{To Sergei Vladimirovich Konyagin on the occasion of his 60th 
birthday}

\begin{abstract}
We find the nearly optimal size of a set $A\subset [N] := \{1,\ldots,N\}$
so that the product set $AA$ satisfies either (i) $|AA| \sim |A|^2/2$
or (ii) $|AA| \sim |[N][N]|$.  This settles problems recently posed in a paper of 
 Cilleruelo, Ramana and Ramar\'e.
\end{abstract}

\footnote{Research of the author supported in part by individual NSF grant DMS-1501982.
Some of this work was carried out at MSRI, Berkeley during the Spring semester of 2017, partially supported by NSF grant DMS-1440140.
}

\maketitle

%%%%%%%%%%%%%%%%%%%%%%%%%

{\Large \section{Introduction}}

For $A,B\subset \NN$ let $AB$ denote the product set $\{ab: a\in A,b\in B\}$.
In the special case $[N]=\{1,2,3\ldots,N\}$, denote by $M_N=|[N][N]|$ the number
of distinct products in an $N$ by $N$ multiplication table.
In a recent paper \cite{CRR} of Cilleruelo, Ramana and Ramar\'e (see also
Problems 15,16 in \cite{INTEGERS}), the following problems were posed:
\begin{enumerate}
\item{} \cite[Problem 1.2]{CRR}.  If $A\subset [N]$ and $|AA|\sim |A|^2/2$, is 
$|A|=o(N/\log^{1/2} N)?$
\item{} \cite[Problem 1.4]{CRR}.  If $A\subset [N]$ and $|AA|\sim M_N$,
is $|A| \sim N$ ?
\end{enumerate}

In this note, we answer both questions in the negative.  Our results are based on a careful analysis of the structure of $[N][N]$
developed in \cite{Ford-Hxyz} and \cite{Ford-Hxy2y}.
Let
\be\label{theta}
\theta=\frac12-\frac{1+\log\log 2}{\log 4} = 1 - \frac{1+\log\log 4}{\log 4} = 0.04303566\ldots
\ee
From \cite{Ford-Hxyz}, we have
\be\label{NN}
M_N \order \frac{N^2}{(\log N)^{2\theta} (\log\log N)^{3/2}}.
\ee
In light of the elementary inequalities $|AA| \le \min(|A|^2,M_N)$, it
follows that if $|AA|\sim \frac12 |A|^2$, then $|A|$ cannot be of order
larger than $M_N^{1/2}$, and if $|AA| \sim M_N$, then $|A|$ cannot have order
of growth smaller than $M_N^{1/2}$.  As we shall see, $M_N^{1/2}$ turns out to be
close the
the threshhold value of $|A|$ for each of these properties to hold.

\begin{thm}\label{Prob1.2thresh}
Let $D>7/2$.
For each $N\ge 10$ there is a set $A\subset [N]$ of size
\[
|A| \ge \frac{N}{(\log N)^\theta (\log\log N)^D},
\]
for which $|AA| \sim |A|^2/2$ as $N\to \infty$.
\end{thm}

Consequently, the largest size $T_N(\eps)$ of a set $A$
with $|AA| \ge (1-\eps) |A|^2/2$  satisfies 
\[
\frac{N}{(\log N)^\theta (\log\log N)^{7/2+o(1)}} \ll T_N(\eps) \ll \frac{N}{(\log N)^\theta (\log\log N)^{3/4}}.
\]

\begin{thm}\label{Prob1.4thresh}
For each $N\ge 10$ there is a set $A\subset [N]$ of size
\[
|A| \le \frac{N}{(\log N)^{\theta}} \exp \left\{ (2/3) \sqrt{\log\log N\log\log\log N} \right\},
\]
for which $|AA| \sim M_N$ as $N\to \infty$.
\end{thm}

The construction of extremal sets satisfying the required properties in either
Theorem \ref{Prob1.2thresh} or \ref{Prob1.4thresh} requires an analysis of
the structure of integers in the ``multiplication table'' $[N][N]$, as worked out
in \cite{Ford-Hxyz}.  From this work, we know that most elements of $[N][N]$ have
$\frac{\log\log N}{\log 2}+O(1)$ prime factors, and moreover, these prime factors 
are not ``compressed at the bottom'', meaning that for most $n\in [N][N]$ we have
\[
\# \{p|n: p\le t \} \le \frac{\log\log t}{\log 2} + O(1) \qquad (3\le t\le N).
\]
Here the terms $O(1)$ should be interpreted as being bounded by
a sufficiently large constant
$C=C(\epsilon)$, where $\epsilon$ is the relative density of exceptional
elements of  $[N][N]$.   This suggests that candidate extremal sets $A$ should 
consist of integers with about half as many prime factors; 
that is, $\omega(n)\approx \frac{\log\log N}{\log 4}$.

In a sequel paper, we will refine the estimates in Theorems \ref{Prob1.2thresh}
and \ref{Prob1.4thresh}.  In particular, we will show that the threshold size of $A$ for the property $|AA| \sim |A|^2/2$ is genuinely smaller than the threshold size 
of $|A|$ for the property $|AA| \sim M_N$.  More precisely, we will show that 
if $|A| \le \frac{N}{(\log N)^\theta} \exp\{O(\sqrt{\log\log N})\}$, then 
 $|AA| \not\sim M_N$.  The proof requires a much more intricate analysis of the
 arguments in the papers \cite{Ford-Hxyz} and \cite{Ford-Hxy2y}.

\textbf{Acknowledgements.}  The author is grateful to Sergei Konyagin for bringing
the paper \cite{CRR} to his attention, and for helpful conversations.

%%%%%%%%%%%%%%%%%%%%%%%%%%%%%%%%%%%%%%%%%%%%%%%%%%%%%%%%%%%%%%%%%%%%%%%%%%
%%%%%%%%%%%%%%%%%%%%%%%%%%%%%%%%%%%%%%%%%%%%%%%%%%%%%%%%%%%%%%%%%%%%%%%%%%

{\Large \section{Preliminaries}}

Here $\omega(n)$ is the number of distinct prime factors of $n$, 
$\omega(n,t)$ is the number of prime factors $p|n$ with $p\le t$,
$\Omega(n)$ is the number of  prime power divisors of $n$, 
$\Omega(n,t)$ is the number of prime powers $p^a|n$ with $p\le t$.
We analyze the distribution of these functions using a simple, but powerful 
technique known as the parametric method (or the ``tilting method'' in
probability theory).

For brevity, we use the notation $\log_k x$ for the $k$-th iterate of the logarithm of $x$.

\begin{lem}\label{weights}
Let $f$ be a real valued multiplicative function such that $0\le f(p^a) \le 1.9^a$
for all primes $p$ and positive integers $a$.  Then, for all $x>1$ we have
\[
\sum_{n\le x} f(n) \ll \frac{x}{\log x} \exp \Big( \sum_{p\le x} \frac{f(p)}{p} \Big).
\]
\end{lem}

\begin{proof}
This is a corollary of a more general theorem of Halberstam and Richert; see
Theorem 01 of \cite{Divisors} and the following remarks.
\end{proof}

In the special case $f(n)=\lambda^{\Omega(n)}$,
where $0<\lambda\le 1.9$,  we get by Mertens' estimate
the uniform bound
\be\label{lambda}
\sum_{n\le x}\lambda^{\Omega(n,t)} \ll x (\log t)^{\lambda-1}.
\ee
This is useful for bounding the tails of the distribution of $\Omega(n,t)$.

%%%%%%%%%%%%%%%%%%%%%%%%%%%%%%%%%%%%%%%%%%%%%%%%%%%%%%%%5
%
%
%
\textbf{\Large
\section{ Proof of Theorem \ref{Prob1.2thresh}}
}

Define
\[
k = \fl{\frac{\log_2 N}{\log 4}}
\]
and let
\[
B = \left\{N/2 < m\le N : m\text{ squarefree}, \; \omega(m)=k,\;  \omega(m,t) \le \frac{\log_2 t}{\log 4}+2 \; (3\le t\le N) \right\}.
\]
Our proof of Theorem \ref{Prob1.2thresh} has three parts:
\begin{enumerate}
 \item[(i)] establish a lower bound
on the size of $B$, showing that the upper bound on $\omega(n,t)$ affects
the size of $B$ only mildly;
\item[(ii)] give an upper bound on the multiplicative energy $E(B)$, which shows that there are
few nontrivial solutions of $b_1b_2=b_3b_4$; consequently, the product set $BB$ 
is large; and
\item[(iii)]  select a thin random  subset $A$ of $B$ that has the desired properties, an idea borrowed from Proposition 3.2 of \cite{CRR}. 
\end{enumerate}

\begin{lem}\label{Bsize}
We have
\[
|B| \gg \frac{N}{(\log N)^{\theta}(\log_2 N)^{3/2}}.
\]
\end{lem}

\begin{lem}\label{energy}
Let $E(B)=|\{(b_1,b_2,b_3,b_4)\in B^4: b_1b_2=b_3b_4\}|$ be the multiplicative
energy of $B$.  Then
\[
E(B) \ll |B|^2 (\log_2 N)^4.
\]
\end{lem}

\begin{lem}\label{randomsubset}
Given $B\subset [N]$ with $E(B)\le |B|^2 f(N)$ and $f(N)\le |B|^{1/2}$,  let $A$ be a subset of $B$ where the
elements of $A$ are chosen at random, each element $b\in B$ chosen with probability
$\rho$ satisfying $\rho^2=o(1/f(N))$ and $\rho|B|^2 \gg |N|^{1.1}$ as $N\to \infty$.  Then with probability $\to 1$ as
$N\to \infty$, we have $|A|\sim \rho |B|$ and $|AA| \sim \frac12 |A|^2$.
\end{lem}

Assuming these three lemmas, it is easy to prove Theorem \ref{Prob1.2thresh}.
We apply Lemma \ref{randomsubset} with $f(N)=C(\log_2 N)^4$, invoking
the energy estimate from Lemma \ref{energy} and the size bound from 
Lemma \ref{Bsize}.   For any function $g(N)\to \infty$ as $N\to \infty$, we take 
\[
\rho = \frac{1}{(\log_2 N)^2 g(N)}
\]
and deduce that there is a set $A\subset [N]$ of size 
\[
|A| \sim \rho |B| \gg \frac{N}{(\log N)^{\theta} (\log_2 N)^{7/2} g(N)},
\]
such that $|AA| \sim \frac12 |A|^2$.

Now we prove the three lemmas.

\begin{proof}[Proof of Lemma \ref{Bsize}]
If $p_j(m)$ denotes the $j$-th smallest (distinct) prime factor of $m$, for 
$1\le j\le \omega(m)$, then the condition  $\omega(m,t) \le \frac{\log_2 t}{\log 4}+2
\; (3\le t\le N)$ is implied by
\[
\log_2 p_j(m) \ge (j-2)\log 4  \quad(1\le j\le \omega(m)).
\]
Indeed, the assertion is trivial if $t<p_1(m)$ since in  this case
$\omega(m,t)=0$.  
If  $p_1(m)\le t\le N$, set $j=\max\{ i : t\ge p_i(m) \}$.  Then
\[
\omega(m,t) = j \le \frac{\log_2 p_j(m)}{\log 4} + 2 \le \frac{\log_2 t}{\log 4}+2.
\]
Thus,
\[
|B| \ge | \{N/2< m\le N :\omega(m)=k, m \text{ squarefree},\; \log_2 p_j(m) \ge j\log 4 - 2\log 4 \; (1\le j\le \omega(m))\}|.
\]

This is closely related to the quantity 
\[
N_{k}(x;\a,\b) = |\{ m\le x: \omega(m)=k, \log_2  p_j(m) \ge \a j - \b (1\le j\le k) \}|,
\]
as defined in \cite{Ford-Smirnov}.  
In fact, the lower bound in \cite[Theorem 1]{Ford-Smirnov} for $N_k(x;\a,\b)$
is proved under the additional conditions that $m$ is squarefree and lies in a dyadic range (\cite[\S 4]{Ford-Smirnov}), although this is not stated explicitly.  Thus, the proof of \cite[Theorem 1]{Ford-Smirnov} applies to 
lower-bounding $|B|$.
In the notation of \cite{Ford-Smirnov}, we have
\[
k=\fl{\frac{\log_2 N}{\log 4}}, \; A=\frac{1}{\log 4}, \; \a=\log 4, \; \b=2\log 4, \;
u=2,\; v=\frac{\log_2 N}{\log 4},\; w=\frac{\log_2 N}{\log 4}-k+3.
\]
Taking $\eps=0.1$, one easily verifies the required conditions for 
\cite[Theorem 1]{Ford-Smirnov}:
\[
\a-\b \le A, \quad w\ge 1+\eps, \quad e^{\a(w-1)} - e^{\a(w-2)} \ge 1+\eps. 
\]
Hence, by the proof of the aforementioned theorem, we obtain
\[
|B| \gg \frac{N (\log_2 N)^{k-2}}{(\log N)(k-1)!},
\]
from which the conclusion follows by Stirling's formula.
\end{proof}

\begin{proof}[Proof of Lemma \ref{energy}]
Set
\[
\b_{13} = \gcd(b_1,b_3), \; \b_{14} = \gcd(b_1,b_4), \; \b_{23} = \gcd(b_2,b_3), \; \b_{24} = \gcd(b_2,b_4),
\]
so that
\[
b_1 = \b_{13} \b_{14}, \quad b_2 =\b_{23} \b_{24}, \quad b_3 = \b_{13} \b_{23}, \quad
b_4 = \b_{14} \b_{24}.
\]
Since $1/2 \le b_1/b_4 \le 2$,
it follows that $1/2 \le \b_{13}/\b_{24} \le 2$ and likewise 
that $1/2 \le \b_{14}/\b_{23} \le 2$.
By reordering variables, we may assume without loss of generality that 
$\min(\b_{13},\b_{24}) \gg N^{1/2}$.  For some parameter $T$, which is 
a power of 2 and satisfies $T=O(N^{1/2})$, we have
\be\label{beta14T}
T \le \b_{14} < 2T.
\ee
This implies that $T/2 \le \b_{23} \le 4T$ and $N/8T \le \b_{13},\b_{24} \le 2N/T$.
We also note that
\be\label{kT}
\omega(b_j,4T)=\Omega(b_j,4T) \le z_T \; (1\le j\le 4), \qquad z_T = \frac{\log_2 (4T)}{\log 4}+2.
\ee
Let $\lam_1,\lam_2\in (0,1)$ be two parameters to be chosen later.
Let $U_T$ be the number of solutions of
\[
b_1b_2 = b_3b_4 \qquad (b_j\in B)
\]
also satisfying \eqref{beta14T}.   Using \eqref{kT}, we see that
\bal
U_T &\le \sum_{\b_{14},\b_{23}\le 4T} \, \sum_{\b_{24}, \b_{13} \le 2 N/T}
\prod_{j=1}^2 \lam_1^{\Omega(\b_{j3}\b_{j4},4T)-z_T}\lam_2^{\Omega(\b_{j3}\b_{j4})-k} 
\prod_{j=3}^4 \lam_1^{\Omega(\b_{1j}\b_{2j},4T)-z_T}\lam_2^{\Omega(\b_{1j}\b_{2j})-k}\\
&= \lam_1^{-4z_T} \lam_2^{-4k}  \sum_{\b_{14},\b_{23}\le 4T} \, \sum_{\b_{24}, \b_{13} \le 2N/T}
\lam_1^{2\Omega(\b_{14}\b_{23})+2\Omega(\b_{13}\b_{24},4T)}\lam_2^{2\Omega(\b_{13}\b_{14} \b_{23}\b_{24})}\\
&= \lam_1^{-4z_T} \lam_2^{-4k} \Bigg( \sum_{\b \le 4T} (\lam_1^2\lam_2^2)^{\Omega(\b)} \Bigg)^2 \Bigg( \sum_{\b\le 2N/T} \lam_1^{2\Omega(\b,4T)} \lam_2^{2\Omega(\b)} \Bigg)^2.
\eal
An application of  Lemma \ref{weights} yields
\bal
U_T &\ll \lam_1^{-4z_T} \lam_2^{-4k} \(T(\log T)^{\lam_1^2\lam_2^2-1} \)^2 \(\frac{N}{T} (\log N)^{\lam_2^2-1} (\log T)^{\lam_1^2 \lam_2^2 - \lam_2^2} \)^2 \\
&= \lam_1^{-4z_T} \lam_2^{-4k}  N^2 (\log N)^{2\lam_2^2-2} (\log T)^{4\lam_1^2\lam_2^2 - 2\lam_2^2-2}.
\eal
We optimize by taking $\lam_1^2=\frac12$ and $\lam_2^2 = \frac{1}{\log 4}$, so that
\[
U_T \ll \frac{N^2}{(\log N)^{2\theta} (\log T)}.
\]
Summing over $T=2^r \ll N^{1/2}$ yields
\[
E(B) \ll \frac{N^2 \log_2 N}{(\log N)^{2\theta}} \ll |B|^2 (\log_2 N)^4,
\]
using Lemma \ref{Bsize}.
\end{proof}

\begin{proof}[Proof of Proposition \ref{randomsubset}]
This is similar to the proof of  Proposition 3.2 of \cite{CRR}.  First, if elements of
$A$ are chosen from $B$ with probability $\rho$, then by easy first and second moment
calculations,
\[
\E |A| = \rho |B|, \qquad \E (|A|-\rho |B|)^2 = O(\rho |B|),
\]
where $\E$ denotes expectation.
By Chebyshev's inequality, $|A|\sim \rho |B|$ with probability
tending to 1 as $N\to \infty$.  By the proof of Proposition 3.2 of
 of \cite{CRR}, we also have
 \[
 \E |AA| = \sum_x \(1 - (1-\rho^2)^{\tau_B(x)/2}\) + O(\rho N),
 \]
 where 
 \[
 \tau_B(x) = | \{ x=b_1 b_2: b_1,b_2 \in B \} |.
 \]
 Now $(1-z)^k=1-kz+O((kz)^2)$ uniformly for $0\le z\le 1$ and $k\ge 1$, and so
 \bal
 \E |AA| &= (\rho^2/2) \sum_x \tau_B(x) + O\Big( \rho^4 \sum_x \tau^2_B(x) \Big) + O(\rho N) \\
 &= (\rho^2/2) |B|^2 + O( \rho^4 E(B)+\rho N )\\
 &= \(\frac12+o(1) \) (\rho|B|)^2.
 \eal
Since $|A|\sim \rho |B|$ with probability tending to 1 as $N\to\infty$, and also $|AA| \le \frac12 |A|(|A|+1)$ for all $|A|$, we conclude that $|AA| \sim \frac12 |A|^2$ with probability tending to 1 as $N\to \infty$.
\end{proof}

%%%%%%%%%%%%%%%%%%%%%%%%%%%%%%%%%%%%%%%%%%%%%%%%%%%%%%%%%%

{\Large
\section{Proof of Theorem \ref{Prob1.4thresh}}}

Again let
\[
k = \fl{\frac{\log_2 N}{\log 4}}.
\]
Define
\[
A = \{ m\le N: \Omega(m) \le k + r \}, \quad r= 2\sqrt{\log_2 N \log_3 N}.
\]
By \eqref{lambda}, we have the size bound
\[
|A| \le \sum_{m\le N} \pfrac{1}{\log 4}^{\Omega(m)-(k+r)} \ll
 \frac{N(\log 4)^r}{(\log N)^\theta} \ll  \frac{N}{(\log N)^\theta} \exp\{
 (2/3) \sqrt{\log_2 N \log_3 N} \}
\]
using \eqref{theta}.
Next, we show that $|AA| \sim M_N$.  Let $B=[N]\setminus A$.  It suffices to show that
\[
|B[N]| \le |AB| + |BB| = o(M_N).
\]
Let $c=ab$, where $a\le N$ and $b\in B$, and 
consider two cases: (i) $\Omega(c) > 2k + h$, where $h=\fl{5\log_3 N}$, and (ii) $\Omega(c) < 2k+h$.  We then have $|B[N]| \le D_1+D_2$, where $D_1$ is the number of
integers $c\le N^2$ with $\Omega(c)>2k+h$, and $D_2$ is the number of pairs
$(a,b)\in [N]^2$ with $\Omega(ab)\le 2k+h$ and $\Omega(b)\ge k+r$.  We will show that each of these is small, essentially by exploiting the imbalance in 
prime factors of $a$ and $b$ implied in the conditions on $D_2$.
  By \eqref{lambda} and \eqref{theta},
\[
D_1\le \sum_{c\le N^2}\pfrac{1}{\log 2}^{\Omega(c)-(2k+h)}   
\ll \frac{N^2}{(\log N)^{2\theta} (1/\log 2)^h} = o(M_N),
\]
in light of estimate \eqref{NN}.  Next,
choose parameters $0<\lambda_2 < 1 < \lambda_1 < 1.9$. Then
\[
D_2\le \sum_{a,b\le N} \lam_2^{\Omega(ab)-(2k+h)} \lam_1^{\Omega(b)-(k+r)} \ll
\lam_1^{-(k+r)}\lam_2^{-(2k+h)} N^2 (\log N)^{\lam_2 + \lam_1\lam_2-2},
\]
invoking \eqref{lambda} again.  A near-optimal choice for the parameters is
\[
\lam_2 = \frac{1-x}{\log 4}, \quad \lam_1 = \frac{1+x}{1-x}, \quad x = \frac{r\log 4}{\log_2 N}.
\]
A little algebra reveals that the previous upper bound on $D_2$ is bounded by
\[
\ll N^2 (\log N)^{-2\theta-\frac{1}{\log 4}\( (1+x)\log(1+x) + (1-x)\log(1-x) \) - \frac{h}{\log_2 N}\log\frac{1-x}{\log 4} }.
\]
By Taylor's expansion,
\[
(1+x)\log(1+x) + (1-x)\log(1-x) \ge x^2 \quad (|x|<1)
\]
and therefore the exponent of $\log N$ is at most 
\[
-2\theta - \frac{x^2}{\log 4} + \frac{h \log\log (4+o(1))}{\log_2 N}
\le -2\theta - 4 \log 4\, \frac{\log_3 N}{\log_2 N} + 1.7 \, \frac{\log_3 N}{\log_2 N} \le -2\theta - 3.8 \, \frac{\log_3 N}{\log_2 N}.
\]
We get that $D_2 \ll N^2(\log N)^{-2\theta} (\log_2 N)^{-3.8} = o(M_N)$ 
and  Theorem \ref{Prob1.4thresh} follows.

%%%%%%%%%%%%%%%%%%%%%%%%%%%%%%%%%%%%%%%%%%%%%%%%%%%%%%%%%%

%%%%%%%%%%%%%%%%%%%%%%%%%

\bibliography{prodset}
\bibliographystyle{alpha}

\end{document}